\documentclass[12pt,reqno]{amsart}%
\usepackage{amsmath, color}
\usepackage[margin=1.5in]{geometry}
\usepackage{graphicx}
\usepackage{amsfonts}
\usepackage{amssymb}%
\usepackage[utf8]{inputenc}
\usepackage{comment}
\usepackage{hyperref}
\usepackage{mathtools}

\newtheorem{theorem}{Theorem}[section]
\theoremstyle{plain}

\newtheorem{corollary}{Corollary}[section]

\newtheorem{lemma}{Lemma}[section]

\newtheorem{proposition}{Proposition}[section]

\numberwithin{equation}{section}
\theoremstyle{definition}

\theoremstyle{remark}
\newtheorem{remark}{Remark}[section]
\allowdisplaybreaks

\def\pd{\partial}
\def\re{\mathbb{R}}
\def\com{\mathbb{C}}

\def\ti{\tilde}
\def\mbb{\mathbb}

\newcommand{\eqal}[1]{\begin{equation}\begin{aligned}#1\end{aligned}\end{equation}}

\newcommand{\osc}{\mathrm{osc}}

\title[Lagrangian mean curvature type equations]{Hessian estimates for shrinkers, expanders, translators, and rotators of the Lagrangian Mean Curvature Flow
}

\author{Arunima Bhattacharya and Jeremy Wall}

\address{Department of Mathematics, Phillips Hall\\
 the University of North Carolina at Chapel Hill, NC }
\email{arunimab@unc.edu}

\address{Department of Mathematics, Phillips Hall\\
 the University of North Carolina at Chapel Hill, NC }
\email{jwall2@unc.edu}

\begin{document}

\maketitle

\begin{abstract}
In this paper, we prove interior Hessian estimates for shrinkers, expanders, translators, and rotators of the Lagrangian mean curvature flow under the assumption that the Lagrangian phase is hypercritical. We further extend our results to a broader class of Lagrangian mean curvature type equations. 
\end{abstract}

\section{Introduction}
A family of Lagrangian submanifolds $X(x,t):\re^n\times\re\to\mbb C^n$ evolves by \textit{Lagrangian mean curvature flow} if it solves
\eqal{
\label{LMCF}
(X_t)^\bot=\Delta_gX=\vec H,
} 
where $\vec H $ denotes the mean curvature vector of the Lagrangian submanifold. 
 The mean curvature vector of the Lagrangian submanifold $(x,Du(x))\subset \mbb C^n$ is determined by the Lagrangian angle or phase $\Theta$, by Harvey-Lawson \cite[Proposition 2.17]{HL}. The Lagrangian angle is given by 
 \begin{equation}\Theta=\sum_{i=1}^n \arctan \lambda_i, \label{sl}
 \end{equation}
 where $\lambda_i$ are the eigenvalues of the Hessian $D^2u$. This angle acts as the potential of the mean curvature vector 
\eqal{
\label{mean}
\vec H=J\nabla_g\Theta,
}
where $g=I_n+(D^2u)^2$ is the induced metric on $(x,Du(x))$, and $J$ is the almost complex structure on $\mbb C^n$. Thus, equation \eqref{sl} is the potential equation for prescribed \textit{Lagrangian mean curvature}. When the Lagrangian phase $\Theta$ is constant, $u$ solves the \textit{special Lagrangian equation} of Harvey-Lawson \cite{HL}. In this case, $H=0$, and $(x,Du(x))$ is a volume-minimizing Lagrangian submanifold.

After a change of coordinates, one can locally write $X(x,t)=(x,Du(x,t))$, such that $\Delta_gX=(J\bar\nabla\Theta(x,t))^\bot$, where $\bar\nabla=(\pd_x,\pd_y)$ is the ambient gradient. This means a local potential $u(x,t)$ evolves by the parabolic equation
\eqal{
\label{ut}
&u_t=\sum_{i=1}^n\arctan\lambda_i,\\
&u(x,0):=u(x).
}

Symmetry reductions of \eqref{LMCF} reduce \eqref{ut} to an elliptic equation for $u(x)$. This is illustrated, for instance, in the work of Chau-Chen-He \cite{CCH}. These solutions model singularities of the mean curvature flow.

If $u(x)$ solves
\eqal{
\label{s}
\sum_{i=1}^n\arctan\lambda_i=s_1+s_2(x\cdot Du(x)-2u(x)),
}
then $X(x,t)=\sqrt{1-2s_2t}\,(x,Du(x))$ is a \textit{shrinker} or \textit{expander} solution of \eqref{LMCF}, if $s_2>0$ or $s_2<0$, respectively. The mean curvature of the initial submanifold $(x,Du(x))$ is given by  $H=-s_2X^\bot$. Entire smooth solutions to \eqref{s} for $s_2>0$ are quadratic polynomials, by Chau-Chen-Yuan \cite{CCY}; see also Huang-Wang \cite{HW} for the smooth convex case. The circle $x^2+u'(x)^2=1$ is a closed example of a shrinker $s_2=1,s_1=0$ in one dimension. We refer the reader to the work of Joyce-Lee-Tsui \cite{JLT}, for other non-graphical examples.

If $u(x)$ solves
\eqal{
\label{tran}
\sum_{i=1}^n\arctan\lambda_i=t_1+t_2\cdot x+t_3\cdot Du(x),
}
then $X(x,t)=(x,Du(x))+t(-t_3,t_2)$ is a \textit{translator} solution of \eqref{LMCF}, with constant mean curvature $H=(-t_3,t_2)^\bot$.  For example, in one dimension, the grim reaper curve $(x,u'(x))=(x,-\ln\cos(x))$, for $t_2=1,t_3=t_1=0$.  Entire solutions to \eqref{tran} with Hessian bounds are quadratic polynomials, by Chau-Chen-He \cite{CCH}; see also Ngyuen-Yuan \cite{NY} for entire ancient solutions to \eqref{tran} with Hessian conditions.

The Hamiltonian vector field $A\cdot z=J\bar \nabla\Theta$ has a real potential given by $\Theta(x,y)=\frac{1}{2i}\langle z,A\cdot z\rangle_{\mbb C^n}$ if $A\in SU(n)$ is skew-adjoint.  Since $\exp(tA)\in U(n)$ preserves the symplectic form 
$dz\wedge d\bar z=\sum dz^i\wedge d\bar z^i$, the Hamiltonian flow $X(x,t)=\exp(tA)(x,Du(x))$ is a Lagrangian immersion with $X_t=AX=J\bar \nabla\Theta$.  For $A=r_2J$ and $\Theta(x,y)=r_1+\frac{r_2}{2}|z|^2$, if $u(x)$ solves
\eqal{
\label{rotator}
\sum_{i=1}^n\arctan\lambda_i=r_1+\frac{r_2}{2}(|x|^2+|Du(x)|^2),
}
then $X(x,t)=\exp(r_2tJ)(x,Du(x))$ is a \textit{rotator} solution of \eqref{LMCF}, with mean curvature $H=r_2(JX)^\bot$. The Yin-Yang curve of Altschuler \cite{Alt} is one such example in one dimension. We also refer the reader to the notes of Yuan \cite[pg. 3]{YNotes}.

A broader class of equations  of interest that generalize equations \eqref{s}, \eqref{tran}, \eqref{rotator}, among others, are the 
\textit{Lagrangian mean curvature type equations}
\eqal{
\label{slag}
\sum_{i=1}^n\arctan\lambda_i=\Theta(x,u(x),Du(x)).
}
 The study of Lagrangian mean curvature-type equations is driven by a geometric interest, particularly because of the notable special cases illustrated above; see \cite{Y20, BS2} for a detailed discussion.

 In this paper, we prove interior Hessian estimates for shrinkers, expanders, translators, and rotators of the Lagrangian mean curvature flow and further extend these results to the broader class of Lagrangian mean curvature-type equations. We assume the Lagrangian phase to be hypercritical, i.e. $|\Theta|\geq (n-1)\frac{\pi}{2}$. This results in the convexity of the potential of the initial Lagrangian submanifold. For certain $\Theta=\Theta(x)$, smooth convex solutions were constructed by Wang-Huang-Bao \cite{WHB1} satisfying $Du(\Omega_1)=\Omega_2$ for prescribed uniformly convex smooth domains $\Omega_i$, following Brendle-Warren \cite{BrW} for the constant $\Theta$ case; see also Huang \cite{H15} for a construction using Lagrangian mean curvature flow.\\

\noindent \textbf{Notations. }Before we present our main results, we clarify some terminology.
\begin{itemize}
\item[I.] By $B_R$ we denote a ball of radius $R$ centered at the origin.
\item[II.] We denote the oscillation of $u$ in $B_R$ by $\osc_{B_R}(u)$.
\item[III.] Let $\Gamma_R = B_R\times u(B_R)\times Du(B_R)\subset B_R\times\re\times\re^n$. Let $\nu_1,\nu_2$ be constants such that for $\Theta(x,z,p)$, we have the following structure conditions
\begin{align}
    |\Theta_x|,|\Theta_z|,|\Theta_p|&\leq \nu_1,\label{struct}\\
    |\Theta_{xx}|,|\Theta_{xz}|,|\Theta_{xp}|,|\Theta_{zz}|,|\Theta_{zp}| &\leq \nu_2 \nonumber
\end{align}
for all $(x,z,p)\in\Gamma_R$. In the above partial derivatives, the variables $x,z,p$ are treated as independent of each other. Observe that this indicates that the above partial derivatives do not have any $D^2u$ or $D^3u$ terms.
 \end{itemize}
\medskip

Our main results are the following:

\begin{theorem}\label{main0}
If $u$ is a $C^4$ solution of any of these equations: \eqref{s}, \eqref{tran}, and \eqref{rotator} on $B_{R}(0)\subset \mathbb{R}^{n}$ where $|\Theta|\geq (n-1)\frac{\pi}{2}$, then we have 
\begin{equation*}
    |D^2u(0)|\leq C_1\exp [C_2(\osc_{B_R}(u)/R^2)^{4n-2}]
\end{equation*}
where $C_1$ and $C_2$ are positive constants depending on $n$ and the following: 
\begin{itemize}
    \item[(1)] $s_2$ for \eqref{s}
    \item[(2)] $t_2,t_3$ for \eqref{tran}
    \item[(3)] $r_2$ for \eqref{rotator}.
\end{itemize}
\end{theorem}
\begin{remark}
In the case of equation \eqref{tran}, since there is no gradient dependence in the derivative of the phase, the precise estimate obtained is
    \[
    |D^2u(0)|\leq C_1\exp [C_2(\osc_{B_R}(u)/R^2)^{3n-2}] .
    \]
\end{remark}

\begin{theorem}\label{main1}
Suppose that $u$ is a $C^4$ solution of \eqref{slag} on $B_{R}(0)\subset\re^n$, where $|\Theta|\geq (n-1)\frac{\pi}{2}$, $\Theta(x,z,p)\in C^2(\Gamma_R)$ is partially convex in the $p$ variable, and satisfies the structure conditions given by \eqref{struct}. Then we have
\begin{equation*}
    |D^2u(0)|\leq C_1\exp[C_2(\osc_{B_{R}}(u)/R^2)^{4n-2}] 
\end{equation*}
where $C_1$ and $C_2$ are positive constants depending on $n$, $\nu_1$,  $\nu_2$.
\end{theorem}
\begin{remark}\label{sing}
    From the singular solutions constructed in \cite[(1.13)]{BS2}, it is evident that the Hessian estimates in  Theorem \ref{main1} will not hold without partial convexity of $\Theta$ in the gradient variable $Du$. 
\end{remark}

One application of the above results is that $C^0$ viscosity solutions to \eqref{s},\eqref{tran}, and \eqref{rotator} with $|\Theta|\geq (n-1)\frac{\pi}{2}$ are analytic inside the domain of the solution, as explained in Remark \ref{dbvp}. 

\smallskip

The concavity of the arctangent operator in \eqref{sl} is closely associated with the range of the Lagrangian phase. 
The phase $(n-2)\frac{\pi}{2}$ is called critical because the level set $\{ \lambda \in \mathbb{R}^n \vert \lambda$ satisfying $ \eqref{sl}\}$ is convex only when $|\Theta|\geq (n-2)\frac{\pi}{2}$ \cite[Lemma 2.2]{YY}. The arctangent operator is concave if $u$ is convex. The concavity of the level set is evident for $|\Theta|\geq (n-1)\frac{\pi}{2}$ since that implies $\lambda>0$, making $F$ concave. The phase $|\Theta|\geq (n-1)\frac{\pi}{2}$ is called hypercritical. The phase $|\Theta|\geq (n-2)\frac{\pi}{2}+\delta$ is called supercritical. The phase $|\Theta|\geq (n-2)\frac{\pi}{2}$ is called critical and supercritical. 
For solutions of the special Lagrangian equation with critical and supercritical phase $|\Theta|\geq (n-2)\frac{\pi}{2}$, Hessian estimates have been obtained by Warren-Yuan \cite{WY9,WY}, Wang-Yuan \cite{WaY}; see also Li \cite{Lcomp} for a compactness approach and Zhou \cite{ZhouHess} for estimates requiring Hessian constraints which generalize criticality. The singular $C^{1,\alpha}$ solutions to \eqref{sl} constructed by Nadirashvili-Vl\u{a}du\c{t} \cite{NV} and Wang-Yuan \cite{WdY} show that interior regularity is not possible for subcritical phases $|\Theta|<(n-2)\frac{\pi}{2}$, without an additional convexity condition, as in Bao-Chen \cite{BCconvex}, Chen-Warren-Yuan \cite{CWY}, and Chen-Shankar-Yuan \cite{CSY}, and that the Dirichlet problem is not classically solvable for arbitrary smooth boundary data.  
In \cite{MooneySavin}, viscosity solutions to \eqref{sl} that are Lipschitz but not $C^1$ were constructed.

If the Lagrangian phase varies $\Theta=\Theta(x)$, then there is less clarity. Hessian estimates for convex smooth solutions with $C^{1,1}$ phase $\Theta=\Theta(x)$ were obtained by Warren in \cite[Theorem 8]{WTh}. For $C^{1,1}$ supercritical phase, interior Hessian and gradient estimates were established by Bhattacharya in \cite{AB}. For  $C^{1,1}$ critical and supercritical phase, interior Hessian and gradient estimates were established by Bhattacharya \cite{AB, AB2d} and Bhattacharya-Mooney-Shankar \cite{BMS} (for $C^2$ phase) respectively. See also Lu \cite{Siyuan}. Recently in \cite{Zhou1}, Zhou established interior Hessian estimates for supercritical $C^{0,1}$ phase. For convex viscosity solutions, interior regularity was established for  $C^2$ phase by Bhattacharya-Shankar in \cite{BS1, BS2}. If $\Theta$ is merely in $C^{\alpha}$ and supercritical, counterexamples to Hessian estimates exist as shown in \cite{AB1}.

While our knowledge is still limited when it comes to the variable Lagrangian phase $\Theta(x)$, it narrows even further when the Lagrangian phase is dependent on both the potential and the gradient of the potential of the Lagrangian submanifold, i.e., $\Theta(x,u,Du)$. Applying the integral method of \cite{AB} to the current problem poses numerous challenges. For instance, establishing the Jacobi-type inequality becomes significantly more intricate due to the presence of the gradient term $Du$ in $\Theta$. Consequently, it is by no means a straightforward process to combine the derivatives of $\Theta$ into a single constant term as in \cite{AB}. 
Next, due to the presence of the gradient term in the phase, the Michael-Simon Sobolev inequality cannot be used to estimate the integral of the volume form by a weighted volume of the non-minimal Lagrangian graph. We circumvent this issue by using the Lewy-Yuan rotation \cite[p.122]{YY}, which is reminiscent of the technique used in \cite{CWY}. This rotation results in a uniformly elliptic Jacobi inequality on the rotated Lagrangian graph, which allows the use of a local maximum principle \cite[Theorem 9.20]{GT}. However, the constants appearing in our Jacobi inequality are dependent on the oscillation of the potential. Therefore we need an explicit dependence of the constants arising in the local maximum principle on $\osc (u)$. To address this, we state and prove a version of the local maximum principle \cite[Theorem 9.20]{GT} applied to our specific equation (see Appendix).
Next, rotating back to the original coordinates and keeping track of the constants appearing at each step, we bound the slope of the gradient graph $(x,Du(x))$ at the origin by an exponential function of the oscillation of $u$.   Note that since the Michael-Simon mean value \cite[Theorem 3.4]{MS} and Sobolev inequalities \cite[Theorem 2.1]{MS} are not employed, there is no explicit dependence on the mean curvature bound in our final estimate.

The critical and supercritical phase case $|\Theta|\geq (n-2)\frac{\pi}{2}$ introduces new challenges requiring new techniques, which we present along with the supercritical phase case $|\Theta|\geq (n-2)\frac{\pi}{2}+\delta$ in forthcoming work \cite{BWall}. \\

\noindent \textbf{Acknowledgments.} AB is grateful to Y. Yuan for helpful discussions.
AB acknowledges
the support of the Simons Foundation grant MP-TSM-00002933 and funding provided by the Bill Guthridge distinguished professorship fund. JW acknowledges
the support of the NSF RTG DMS-2135998 grant.

\section{Preliminaries}

For the convenience of the readers, we recall some preliminary results. We first introduce some notations that will be used in this paper.
The induced Riemannian metric on the Lagrangian submanifold $X=(x,Du(x))\subset \mathbb{R}^n\times\mathbb{R}^n$ is given by
\[g=I_n+(D^2u)^2 .
\]
We denote
 \begin{align*} 
    \partial_i=\frac{\partial}{\partial x_i} \text{ , }
     \partial_{ij}=\frac{\partial^2}{\partial x_i\partial x_j} \text{ , }
     u_i=\partial_iu \text{ , }
    u_{ij}=\partial_{ij}u.
    \end{align*}
  Note that for the functions defined below, the subscripts on the left do not represent partial derivatives\begin{align*}
    h_{ijk}=\sqrt{g^{ii}}\sqrt{g^{jj}}\sqrt{g^{kk}}u_{ijk},\quad
    g^{ii}=\frac{1}{1+\lambda_i^2}.
    \end{align*}
Here $(g^{ij})$ is the inverse of the matrix $g$ and $h_{ijk}$ denotes the second fundamental form when the Hessian of $u$ is diagonalized.
The volume form, gradient, and inner product with respect to the metric $g$ are given by
\begin{align*}
    dv_g=\sqrt{\det g}dx &= Vdx \text{ , }\qquad
    \nabla_g v=g^{ij}v_iX_j,\\
    \langle\nabla_gv,\nabla_g w\rangle_g &=g^{ij}v_iw_j \text{ , }\quad
    |\nabla_gv|^2=\langle\nabla_gv,\nabla_g v\rangle_g.
\end{align*}

Next, we derive the Laplace-Beltrami operator on the non-minimal submanifold $(x, Du(x))$. Taking variations of the energy functional $\int |\nabla_g v|^2 dv_g$ with respect to $v$, one gets the Laplace-Beltrami operator of the metric $g$:
\begin{align}
\Delta_g  =\frac{1}{\sqrt{ g}}\partial_i(\sqrt{ g}g^{ij}\partial_j )
&=g^{ij}\partial_{ij}+\frac{1}{\sqrt{g}}\partial_i(\sqrt{g}g^{ij})\partial_j \label{2!}\\
&=g^{ij}\partial_{ij}-g^{jp}u_{pq}(\pd_q\Theta)  \partial_j. \nonumber
\end{align} The last equation follows from the following computation:

\begin{align}
    \frac{1}{\sqrt{g}}\partial_i(\sqrt{g}g^{ij})&=\frac{1}{\sqrt{g}}\partial_i(\sqrt{g})g^{ij}+\partial_ig^{ij} \nonumber \\
    &=\frac{1}{2}(\partial_i \ln g)g^{ij}+\partial_kg^{kj}\nonumber\\
    &=\frac{1}{2}g^{kl}\partial_i g_{kl}g^{ij}-g^{kl}\partial_k g_{lb}g^{bj}\nonumber\\
    &=-g^{jp}g^{ab}u_{abq}u_{pq}
    =-g^{jp} u_{pq}\pd_q\Theta\label{lolz}
\end{align}
where the last equation follows from \eqref{111} and \eqref{linearize} below.
The first
derivative of the metric $g$ is given by 
\begin{align}
    \partial_i g_{ab}=\partial_i(\delta_{ab}+u_{ak}u_{kb})=u_{aik}u_{kb}+u_{bik}u_{ka}\overset{\text{at } x_0}{=}u_{abi}(\lambda_a+\lambda_b), \label{111}
    \end{align}
    assuming the Hessian of $u$ is diagonalized at $x_0$.
On taking the gradient of both sides of the Lagrangian mean curvature type equation \eqref{slag}, we get
\begin{equation}
\sum_{a,b=1}^{n}g^{ab}u_{jab}=\pd_j\Theta (x,u(x),Du(x)).\label{linearize}
\end{equation}

For the general phase $\Theta(x,u(x),Du(x))$, assuming the Hessian $D^2u$ is diagonalized at $x_0$, we get
\begin{align}
    \pd_i \Theta(x,u(x),Du(x)) &= \Theta_{x_i} + \Theta_u u_i + \sum_{k=1}^n \Theta_{u_k}u_{ki} \label{dipsi}\\
    &\overset{x_0}{=} \Theta_{x_i} + \Theta_u u_i + \Theta_{u_i}\lambda_i. \label{dipsi@p}
\end{align}

So from \eqref{dipsi@p} and \eqref{mean}, we get, at the point $x_0\in B_R$,
\begin{align}
|\vec H|_g^2=g^{ii}(\pd_i\Theta)^2 &=g^{ii}\bigg(\Theta_{x_i}^2 + \Theta_u^2u_i^2 + \Theta_{u_i}^2\lambda_i^2 + 2\Theta_{x_i}\Theta_{u}u_i + 2\Theta_{x_i}\Theta_{u_i}\lambda_i +2\Theta_{u}\Theta_{u_i}u_i\lambda_i\bigg)\nonumber\\
&\leq 3g^{ii}\bigg(\Theta_{x_i}^2 + \Theta_u^2u_i^2 + \Theta_{u_i}^2\lambda_i^2\bigg)\nonumber\\
&\leq C(\nu_1, n, \osc_{B_{R+1}}(u)).\nonumber
\end{align}

Taking the $j$-th partial derivative of \eqref{dipsi}, we get
\begin{align}
    \pd_{ij}\Theta(x,u(x),Du(x)) &= \Theta_{x_ix_j} + \Theta_{x_i u}u_j + \sum_{r=1}^n \Theta_{x_iu_r}u_{rj}\nonumber\\
    & \qquad +\left(\Theta_{ux_j} + \Theta_{uu}u_j + \sum_{s=1}^n \Theta_{u u_s}u_{sj} \right)u_i + \Theta_u u_{ij}\nonumber\\
    & \qquad +\sum_{k=1}^n \left(\Theta_{u_kx_j} + \Theta_{u_ku}u_j + \sum_{\ell=1}^n \Theta_{u_ku_\ell}u_{\ell j}\right)u_{ki}+\sum_{k=1}^n \Theta_{u_k}u_{kij}\nonumber\\
    & \overset{x_0}{=} \Theta_{x_ix_j} + \Theta_{x_i u}u_j + \Theta_{x_iu_j}\lambda_j\label{dijpsi@p}\\
    & \qquad +\left(\Theta_{ux_j} + \Theta_{uu}u_j + \Theta_{u u_j}\lambda_j \right)u_i + \Theta_u \lambda_i\delta_{ij}\nonumber\\
    & \qquad +\left(\Theta_{u_ix_j} + \Theta_{u_iu}u_j + \Theta_{u_iu_j}\lambda_j\right)\lambda_i + \sum_{k=1}^n \Theta_{u_k}u_{kij}.\nonumber
\end{align}

Observe that when $\Theta$ is constant, one can choose harmonic co-ordinates $\Delta_g x=0$, which 
 reduces the Laplace-Beltrami operator on the minimal submanifold $\{(x,Du(x))|x\in B_R(0)\}$ to the linearized operator of \eqref{sl} at $u$.

\medskip

\section{The slope as a subsolution to a fully nonlinear PDE}

In this section, we prove a Jacobi-type inequality for the slope of the gradient graph $(x,Du(x))$, i.e., we show that a certain function of the slope of the gradient graph $(x,Du(x))$ is almost strongly subharmonic.

\begin{proposition}\label{PropLaplGrad}
Let $u$ be a $C^4$ convex solution of \eqref{slag} in $\mathbb{R}^{n}$. Suppose that the Hessian $D^{2}u$ is diagonalized at point $x_0$. Then we have the following at $x_0$
\begin{equation*}
     \frac{1}{n}|\nabla_g \log\sqrt{\det g}|^2_g\leq \sum_{i=1}^n\lambda_i^2h_{iii}^2 + \sum_{i\neq j} \lambda_j^2h_{jji}^2
\end{equation*} and
\begin{align}
    \Delta_g \log\sqrt{\det g} &\overset{x_0}{=}\sum_{i=1}^n(1 + \lambda_i^2)h_{iii}^2 + \sum_{j\neq i}(3 + \lambda_j^2 + 2\lambda_i\lambda_j)h_{jji}^2\nonumber\\
    &\qquad + 2\sum_{i<j<k}(3 + \lambda_i\lambda_j + \lambda_j\lambda_k + \lambda_k\lambda_i)h_{ijk}^2\nonumber\\
    &\quad +\sum_{i=1}^n g^{ii}\lambda_i\pd_{ii}\Theta - \sum_{i=1}^n g^{ii}\lambda_i(\pd_i\Theta) \pd_i\log\sqrt{\det g}.\nonumber
\end{align}
\end{proposition}

\begin{proof}
We compute some derivatives of the metric $g$. We have 
\begin{align}
    \pd_j g_{ab} &= \sum_{k=1}^n(u_{akj}u_{kb} + u_{ak}u_{kbj})\nonumber\\
    & \overset{x_0}{=} u_{abj}(\lambda_a + \lambda_b )\label{djg_ab@p}
\end{align}
and

\begin{align}
    \pd_ig^{ab} &= -g^{ak}\pd_ig_{kl}g^{lb}\nonumber\\
    & \overset{x_0}{=} -g^{aa}\pd_ig_{ab}g^{bb}\nonumber\\
    & \overset{x_0}{=} -g^{aa}g^{bb}u_{abi}(\lambda_a + \lambda_b ).\label{dig^ab@p}
\end{align}

Hence 
\begin{align}
    \pd_{ij}g_{ab} &= \sum_{k=1}^n(u_{akji}u_{kb} + u_{akj}u_{kbi} + u_{aki}u_{kbj} + u_{ak}u_{kbij})\nonumber\\
    & \overset{x_0}{=} u_{abji}(\lambda_a + \lambda_b ) +\sum_{k=1}^n(u_{akj}u_{kbi} + u_{aki}u_{kbj}). \nonumber
\end{align}

In order to substitute the 4th order derivatives, we take the partial derivative of \eqref{linearize} and get
\begin{align}
    \sum_{i,j=1}^ng^{ij}u_{ijk\ell} &= \pd_{k\ell}\Theta - \sum_{i,j=1}^n\pd_\ell g^{ij} u_{ijk}\nonumber\\
    &\overset{x_0}{=} \pd_{k\ell}\Theta + \sum_{i,j=1}^ng^{ii}g^{jj}u_{ij\ell }u_{ijk}(\lambda_i + \lambda_j ).\nonumber
\end{align}

Thus, we have
\begin{equation}\label{gD2g}
\sum_{i,j=1}^ng^{ij}\pd_{ij}g_{ab} \overset{x_0}{=} (\lambda_a + \lambda_b)\pd_{ab}\Theta + \sum_{i,j=1}^ng^{ii}g^{jj}u_{ija}u_{ijb}(\lambda_i + \lambda_j )(\lambda_a + \lambda_b) + \sum_{i,k=1}^n2g^{ii}u_{aki}u_{bki}.
\end{equation}

Next, we compute the norm of the gradient:
\begin{align}
    \frac{1}{n}|\nabla_g \log\sqrt{\det g}|^2_g &\overset{x_0}{=}\sum_{i=1}^n\frac{1}{n}g^{ii}\left(\pd_i\log\sqrt{\det g}\right)^2\nonumber\\
    &\overset{x_0}{=}\sum_{i=1}^n\frac{1}{n}g^{ii}\left(\sum_{a,b=1}^n\frac{1}{2}g^{ab}\pd_ig_{ab}\right)^2\nonumber\\
    &\overset{x_0}{=}\sum_{i=1}^n\frac{1}{n}g^{ii}\left(\sum_{a,b=1}^n\frac{1}{2}g^{ab}u_{abi}(\lambda_a + \lambda_b)\right)^2 \quad\text{from \eqref{djg_ab@p}}\nonumber\\
    &\overset{x_0}{=}\sum_{i=1}^n\frac{1}{n}g^{ii}\left(\sum_{a=1}^n g^{aa}u_{aai}\lambda_a\right)^2\label{dilogdetg@p}\\
    &\leq\sum_{i,a=1}^n g^{ii}(g^{aa})^2u_{aai}^2\lambda_a^2\nonumber\\
    &\overset{x_0}{=}\sum_{i,a=1}^n h_{aai}^2\lambda_a^2\nonumber\\
    &\overset{x_0}{=} \sum_{i=1}^n\lambda_i^2h_{iii}^2 + \sum_{i\neq j} \lambda_j^2h_{jji}^2.\nonumber
\end{align}
From here, we need to calculate $\Delta_g \log\sqrt{\det g}$, where again, the Laplace-Beltrami operator takes the form of \eqref{2!}. From the above calculations, we observe that 
\begin{align}
    \sum_{i,j=1}^ng^{ij}\pd_{ij}\log\sqrt{\det g} &= \sum_{i,j=1}^ng^{ij}\pd_j\left(\frac{1}{\sqrt{\det g}}\frac{1}{2\sqrt{\det g}}\pd_i\det g \right)\nonumber\\
    &=\sum_{i,j,a,b=1}^ng^{ij}\pd_j\left(\frac{1}{2\det g}\det g\; g^{ab}\pd_i g_{ab} \right)\nonumber\\
    &=\sum_{i,j,a,b=1}^ng^{ij} \frac{1}{2}\pd_j\left(g^{ab}\pd_i g_{ab}\right)\nonumber\\
    &=\sum_{i,j,a,b=1}^n g^{ij}\frac{1}{2}\left((\pd_j g^{ab})\pd_ig_{ab} + g^{ab}\pd_{ij}g_{ab} \right).\label{dijlogdetg}
\end{align}
Using \eqref{djg_ab@p} and \eqref{dig^ab@p}, we see that the first term of \eqref{dijlogdetg} becomes
\begin{align}
    \sum_{i,j,a,b=1}^n\frac{1}{2}g^{ij}(\pd_j g^{ab})\pd_ig_{ab} 
    \overset{x_0}{=}-\frac{1}{2}\sum_{i,a,b=1}^n g^{ii}g^{aa}g^{bb}u_{abi}^2(\lambda_a + \lambda_b )^2.\label{ldg1}
\end{align}

Using \eqref{gD2g}, the second term of \eqref{dijlogdetg} becomes
\begin{align}
    \sum_{i,j,a,b=1}^n\frac{1}{2}g^{ij}g^{ab}\pd_{ij}g_{ab} 
    \overset{x_0}{=} \sum_{a=1}^n g^{aa}\lambda_a\pd_{aa}\Theta + \sum_{i,j,a=1}^n g^{aa}g^{ii}g^{jj}u_{ija}^2(\lambda_i + \lambda_j)\lambda_a + \sum_{i,k,a=1}^n g^{aa}g^{ii}u_{aki}^2.\label{ldg2}
\end{align}

Combining \eqref{ldg1} and \eqref{ldg2}, we get
\begin{align}
    \sum_{i,j=1}^n g^{ij}\pd_{ij}\log\sqrt{\det g} &\overset{x_0}{=} \sum_{a=1}^n g^{aa}\lambda_a\pd_{aa}\Theta +\sum_{i,j,a=1}^n g^{aa}g^{ii}g^{jj}u_{ija}^2(\lambda_i + \lambda_j)\lambda_a + \sum_{i,k,a=1}^n g^{aa}g^{ii}u_{aki}^2\nonumber\\
    &\qquad-\frac{1}{2} \sum_{i,a,b=1}^n g^{ii}g^{aa}g^{bb}u_{abi}^2 (\lambda_a + \lambda_b )^2\nonumber\\
    &\overset{x_0}{=} \sum_{a=1}^n g^{aa}\lambda_a\pd_{aa}\Theta + \sum_{a,b,c=1}^n g^{aa}g^{bb}g^{cc}u_{abc}^2(\lambda_b + \lambda_c)\lambda_a \nonumber\\
    &\qquad+ \sum_{a,b,c=1}^n g^{aa}g^{bb}g^{cc}u_{abc}^2(1 + \lambda_c^2) -\frac{1}{2} \sum_{a,b,c=1}^n g^{aa}g^{bb}g^{cc}u_{abc}^2 (\lambda_a + \lambda_b )^2\nonumber\\
    &\overset{x_0}{=}\sum_{a=1}^n g^{aa}\lambda_a\pd_{aa}\Theta + \sum_{a,b,c=1}^n h_{abc}^2(1+\lambda_b\lambda_c)\nonumber\\
    &\overset{x_0}{=}\sum_{i=1}^n g^{ii}\lambda_i\pd_{ii}\Theta + \sum_{i=1}^n(1 + \lambda_i^2)h_{iii}^2 + \sum_{j\neq i}(3 + \lambda_j^2 + 2\lambda_i\lambda_j)h_{jji}^2\nonumber\\
    &\qquad + 2\sum_{i<j<k}(3 + \lambda_i\lambda_j + \lambda_j\lambda_k + \lambda_k\lambda_i)h_{ijk}^2.\nonumber
\end{align}

Altogether, we get
\begin{align}
    \Delta_g \log\sqrt{\det g} &\overset{x_0}{=}\sum_{i=1}^n(1 + \lambda_i^2)h_{iii}^2 + \sum_{j\neq i}(3 + \lambda_j^2 + 2\lambda_i\lambda_j)h_{jji}^2\nonumber\\
    &\qquad + 2\sum_{i<j<k}(3 + \lambda_i\lambda_j + \lambda_j\lambda_k + \lambda_k\lambda_i)h_{ijk}^2\nonumber\\
    &\quad +\sum_{i=1}^n g^{ii}\lambda_i\pd_{ii}\Theta - \sum_{i=1}^n g^{ii}\lambda_i(\pd_i\Theta)\pd_i\log\sqrt{\det g}. \nonumber
\end{align}
\end{proof}

\begin{lemma}\label{Jpoint}
Let $u$ be a $C^4$ convex solution of \eqref{slag} in $B_2(0)\subset\mathbb{R}^{n}$ where $\Theta(x,z,p)\in C^2(\Gamma_2)$ is partially convex in the $p$ variable and satisfies \eqref{struct}. Suppose that the Hessian $D^{2}u$ is diagonalized at $x_0\in B_1(0)$. Then at $x_0$, the function $\log\sqrt{\det g}$ satisfies 
\begin{equation}\Delta_g \log\sqrt{\det g}\geq c(n)|\nabla_g\log\sqrt{\det g}|^2-C\label{J}
\end{equation}
where $C=C(n,\nu_1,\nu_2)(1 + (\osc_{B_2}(u))^2)$.
\end{lemma}
\begin{proof}
    \begin{itemize}
    \item [Step 1.] From Proposition \ref{PropLaplGrad}, we get, at $x_0\in B_1(0)$,
\begin{align}
    \Delta_g \log\sqrt{\det g} - \frac{1}{n}|\nabla_g \log\sqrt{\det g}|^2_g &\geq \sum_{i=1}^n(1 + \lambda_i^2)h_{iii}^2 + \sum_{j\neq i}(3 + \lambda_j^2 + 2\lambda_i\lambda_j)h_{jji}^2 \nonumber\\
    &\qquad + 2\sum_{i<j<k}(3 + \lambda_i\lambda_j + \lambda_j\lambda_k + \lambda_k\lambda_i)h_{ijk}^2\nonumber\\
    &\qquad - \sum_{i=1}^n\lambda_i^2h_{iii}^2 - \sum_{i\neq j}\lambda_j^2h_{jji}^2\nonumber\\
    &\qquad +\sum_{i=1}^n g^{ii}\lambda_i\pd_{ii}\Theta - \sum_{i=1}^n g^{ii}\lambda_i(\pd_i\Theta)\pd_i\log\sqrt{\det g} \nonumber\\
    &=\sum_{i=1}^nh_{iii}^2 + \sum_{j\neq i}(3 + 2\lambda_i\lambda_j)h_{jji}^2 \nonumber\\
    &\qquad + 2\sum_{i<j<k}(3 + \lambda_i\lambda_j + \lambda_j\lambda_k + \lambda_k\lambda_i)h_{ijk}^2\nonumber\\
    &\qquad +\sum_{i=1}^n g^{ii}\lambda_i\pd_{ii}\Theta - \sum_{i=1}^n g^{ii}\lambda_i(\pd_i\Theta)\pd_i\log\sqrt{\det g}\nonumber\\
    &\geq \sum_{i=1}^n g^{ii}\lambda_i\pd_{ii}\Theta - \sum_{i=1}^n g^{ii}\lambda_i(\pd_i\Theta)\pd_i\log\sqrt{\det g}\label{LapGradDiff}
\end{align}
where the last inequality follows from the convexity of $u$.

From here, we use \eqref{dijpsi@p} to get
\begin{align}
    \sum_{a=1}^n g^{aa}\lambda_a \pd_{aa}\Theta &\overset{x_0}{=}\sum_{a=1}^n \frac{\lambda_a}{1+\lambda_a^2}\bigg[\Theta_{x_ax_a} + \Theta_{x_a u}u_a +  \Theta_{x_au_a}\lambda_a\nonumber\\
    & \qquad +\left(\Theta_{ux_a} + \Theta_{uu}u_a +  \Theta_{u u_a}\lambda_a \right)u_a + \Theta_u \lambda_a\nonumber\\
    & \qquad +\left(\Theta_{u_ax_a} + \Theta_{u_au}u_a + \Theta_{u_au_a}\lambda_a\right)\lambda_a\nonumber\\
    & \qquad +\sum_{k=1}^n \Theta_{u_k}u_{kaa}\bigg]\nonumber \\
    &\overset{x_0}{=}\sum_{a=1}^n \frac{\lambda_a}{1+\lambda_a^2}\bigg[\Theta_{x_ax_a} + 2\Theta_{x_a u}u_a + 2\Theta_{x_au_a}\lambda_a + 2\Theta_{uu_a}u_a\lambda_a\\
    &\qquad + \Theta_{u}\lambda_a + \Theta_{uu}u_a^2 + \Theta_{u_au_a}\lambda_a^2 + \sum_{k=1}^n \Theta_{u_k}u_{kaa}\bigg] \nonumber\\
    &\overset{x_0}{=}\sum_{a=1}^n \frac{\lambda_a}{1+\lambda_a^2}\bigg[\Theta_{x_ax_a} + 2\Theta_{x_a u}u_a+ 2\Theta_{x_au_a}\lambda_a + 2\Theta_{uu_a}u_a\lambda_a \label{d2psi1}\\
    &\qquad + \Theta_{u}\lambda_a + \Theta_{uu}u_a^2 + \Theta_{u_au_a}\lambda_a^2\bigg] + \sum_{k=1}^n \Theta_{u_k}\pd_k\log\sqrt{\det g}\quad\text{using \eqref{dilogdetg@p}.} \nonumber
\end{align}

Similarly, using \eqref{dipsi@p}, we get
\begin{align}
    \sum_{i=1}^n g^{ii}\lambda_i(\pd_i\Theta)\pd_i\log\sqrt{\det g} &\overset{x_0}{=} \sum_{i=1}^n \frac{\lambda_i}{1+ \lambda_i^2}\left(\Theta_{x_i} + \Theta_u u_i + \Theta_{u_i}\lambda_i\right)\pd_i\log\sqrt{\det g}.
\end{align}

Hence, \eqref{LapGradDiff} becomes
\begin{align}
    \sum_{a=1}^n &g^{aa}\lambda_a\pd_{aa}\Theta - \sum_{i=1}^n g^{ii}\lambda_i(\pd_i\Theta)\pd_i\log\sqrt{\det g}\nonumber\\
    &\overset{x_0}{=}\sum_{a=1}^n\frac{\lambda_a}{1+\lambda_a^2}\bigg[\Theta_{x_ax_a} + 2\Theta_{x_a u}u_a+ 2\Theta_{x_au_a}\lambda_a + 2\Theta_{uu_a}u_a\lambda_a
    + \Theta_{u}\lambda_a + \Theta_{uu}u_a^2 + \Theta_{u_au_a}\lambda_a^2\bigg]\nonumber\\
    &\quad+ \sum_{k=1}^n\Theta_{u_k}\pd_k\log\sqrt{\det g}-\sum_{k=1}^n\frac{\lambda_k}{1+ \lambda_k^2}\left(\Theta_{x_k} + \Theta_u u_k + \Theta_{u_k}\lambda_k\right)\pd_k\log\sqrt{\det g} \nonumber\\
    &\overset{x_0}{=}\sum_{a=1}^n\frac{\lambda_a}{1+\lambda_a^2}\bigg[\Theta_{x_ax_a} + 2\Theta_{x_a u}u_a+ 2\Theta_{x_au_a}\lambda_a + 2\Theta_{uu_a}u_a\lambda_a + \Theta_{u}\lambda_a + \Theta_{uu}u_a^2 + \Theta_{u_au_a}\lambda_a^2\bigg]\label{psiulambdas1}\\
    &\quad+\sum_{k=1}^n\frac{1}{1+ \lambda_k^2}\left(\Theta_{u_k} - \Theta_{x_k}\lambda_k - \Theta_u u_k\lambda_k\right)\pd_k\log\sqrt{\det g}.\label{ldgyoungs}
\end{align}
 
    \item[Step 2.1.]Using Young's inequality, \eqref{ldgyoungs} can be bounded below by
\begin{align}
    \sum_{k=1}^n\frac{1}{1+ \lambda_k^2} &\left(\Theta_{u_k} - \Theta_{x_k}\lambda_k - \Theta_u u_k\lambda_k\right)\pd_k\log\sqrt{\det g} \nonumber\\
    &\geq -\sum_{k=1}^n\frac{1}{1+ \lambda_k^2}\left(|\Theta_{u_k}| + |\Theta_{x_k}|\lambda_k + |\Theta_u u_k|\lambda_k\right)|\pd_k\log\sqrt{\det g}|\nonumber\\
    &\geq -\frac{1}{2\epsilon}\sum_{k=1}^n\frac{1}{1+ \lambda_k^2}\left(\Theta_{u_k}^2 + \Theta_{x_k}^2\lambda_k^2 + \Theta_u^2 u_k^2\lambda_k^2\right) - \frac{\epsilon}{2}|\nabla_g\log\sqrt{\det g}|_g^2. \label{psiulambdas2}
\end{align}

Altogether, from \eqref{LapGradDiff}, \eqref{psiulambdas1}, and \eqref{psiulambdas2}, we have
\begin{align}
    \Delta_g &\log\sqrt{\det g} - \left(\frac{1}{n}-\frac{\epsilon}{2}\right)|\nabla_g \log\sqrt{\det g}|^2_g\nonumber\\
    &\geq \sum_{a=1}^n\frac{\lambda_a}{1+\lambda_a^2}\bigg[\Theta_{x_ax_a} + 2\Theta_{x_a u}u_a + 2\Theta_{x_au_a}\lambda_a + 2\Theta_{uu_a}u_a\lambda_a + \Theta_{u}\lambda_a + \Theta_{uu}u_a^2 + \Theta_{u_au_a}\lambda_a^2\bigg]\nonumber\\
    &\quad -\frac{1}{2\epsilon}\sum_{k=1}^n\frac{1}{1+ \lambda_k^2}\left(\Theta_{u_k}^2 + \Theta_{x_k}^2\lambda_k^2 + \Theta_u^2 u_k^2\lambda_k^2\right).\nonumber
\end{align}
Let $\epsilon = \frac{1}{n}$, so that we achieve
\begin{align}
    \Delta_g &\log\sqrt{\det g} - \frac{1}{2n}|\nabla_g \log\sqrt{\det g}|^2_g\nonumber\\
    &\geq \sum_{a=1}^n\frac{\lambda_a}{1+\lambda_a^2}\bigg[\Theta_{x_ax_a} + 2\Theta_{x_a u}u_a + 2\Theta_{x_au_a}\lambda_a + 2\Theta_{uu_a}u_a\lambda_a + \Theta_{u}\lambda_a + \Theta_{uu}u_a^2 + \Theta_{u_au_a}\lambda_a^2\bigg]\label{psiulambdas3}\\
    &\quad -\frac{n}{2}\sum_{k=1}^n\frac{1}{1+ \lambda_k^2}\left(\Theta_{u_k}^2 + \Theta_{x_k}^2\lambda_k^2 + \Theta_u^2 u_k^2\lambda_k^2\right).\label{psiulambdas4}
\end{align}

\item[Step 2.2] Here, we use the assumption that $\Theta(x,z,p)$ is partially convex in the $p$ variable. That is, $\Theta_{u_a u_a} \geq 0$. This comes from the fact that $D^2_{Du}\Theta$ is a symmetric positive definite matrix. Combined with the fact that $u$ is a convex function, we get
\[
\frac{\lambda_a^3}{1 + \lambda_a^2}\Theta_{u_a u_a} \geq 0.
\]
Thus, \eqref{psiulambdas3} becomes
\begin{align}
    &\sum_{a=1}^n\frac{\lambda_a}{1+\lambda_a^2}\bigg[\Theta_{x_ax_a} + 2\Theta_{x_a u}u_a+ 2\Theta_{x_au_a}\lambda_a + 2\Theta_{uu_a}u_a\lambda_a + \Theta_{u}\lambda_a + \Theta_{uu}u_a^2 + \Theta_{u_au_a}\lambda_a^2 \bigg]\nonumber\\
    &\quad \geq \sum_{a=1}^n\frac{\lambda_a}{1+\lambda_a^2}\bigg[\Theta_{x_ax_a} + 2\Theta_{x_a u}u_a+ 2\Theta_{x_au_a}\lambda_a + 2\Theta_{uu_a}u_a\lambda_a + \Theta_{u}\lambda_a + \Theta_{uu}u_a^2\bigg]\nonumber\\
    &\quad \geq -\sum_{a=1}^n\frac{\lambda_a}{1+\lambda_a^2}\bigg[|\Theta_{x_ax_a}| + 2|\Theta_{x_a u}u_a|+ 2|\Theta_{x_au_a}|\lambda_a + 2|\Theta_{uu_a}u_a|\lambda_a + |\Theta_{u}|\lambda_a + |\Theta_{uu}|u_a^2\bigg].\label{abspsiulambda}
\end{align}

Now, for all $\lambda_a \in [0,\infty]$, we have that
\[
0 \leq \frac{\lambda_a}{1 + \lambda_a^2}\leq 1 \quad \text{and} \quad 0 \leq \frac{\lambda_a^2}{1 + \lambda_a^2}\leq 1.
\]
Hence, \eqref{psiulambdas4} and \eqref{abspsiulambda} yield
\begin{align}
    \Delta_g &\log\sqrt{\det g} - \frac{1}{2n}|\nabla_g \log\sqrt{\det g}|^2_g\nonumber\\
    &\geq - \sum_{a=1}^n\bigg[|\Theta_{x_ax_a}| + 2|\Theta_{x_a u}u_a| + 2|\Theta_{x_au_a}| + 2|\Theta_{uu_a}u_a| + |\Theta_{u}| + |\Theta_{uu}|u_a^2\bigg]\label{abspsiulambda2}\\
    &\quad - \frac{n}{2}\sum_{a=1}^n\left(\Theta_{u_a}^2 + \Theta_{x_a}^2 + \Theta_u^2 u_a^2\right).\nonumber
\end{align}
We observe that \eqref{abspsiulambda2} is bounded by
\begin{align}
    \sum_{a=1}^n &\bigg[|\Theta_{x_ax_a}| + 2|\Theta_{x_a u}u_a|+ 2|\Theta_{x_au_a}| + 2|\Theta_{uu_a}u_a| + |\Theta_{u}| + |\Theta_{uu}|u_a^2\bigg]\nonumber\\
    &\quad + \frac{n}{2}\sum_{a=1}^n\left(\Theta_{u_a}^2 + \Theta_{x_a}^2 + \Theta_u^2 u_a^2\right)\nonumber\\
    &\leq C(n,\nu_1,\nu_2)\left(1 + \sum_{a=1}^n(|u_a| + u_a^2)\right)\nonumber\\
    &\leq C(n,\nu_1,\nu_2)(1 + |Du(x_0)| + |Du(x_0)|^2)\nonumber\\
    &\leq C(n,\nu_1,\nu_2)(1 + ||Du||_{L^\infty(B_1)} + ||Du||_{L^\infty(B_1)}^2)\nonumber\\
    &\leq C(n,\nu_1,\nu_2)(1 + (\osc_{B_2}(u))^2)\nonumber
\end{align}
where the last inequality comes from the convexity of $u$ and Young's inequality.

Therefore,
\begin{equation*}
    \Delta_g \log\sqrt{\det g} - \frac{1}{2n}|\nabla_g \log\sqrt{\det g}|^2_g \geq - C(n,\nu_1,\nu_2)(1 + (\osc_{B_2}(u))^2)
\end{equation*}
as desired.
\end{itemize}
\end{proof}

\begin{corollary}\label{scor}
    Let $u$ be a $C^4$ convex solution to \eqref{s} in $B_2(0)\subset\re^n$. Assuming the Hessian $D^2u$ is diagonalized at $x_0\in B_1(0)$, \eqref{J} holds with $C= C(n,s_2)(1+(\osc_{B_2}(u))^2)$.
\end{corollary}
\begin{proof}
    Let $x_0\in B_1$. As $\Theta(x,u(x),Du(x)) = s_1 + s_2(x\cdot Du(x) - 2u(x))$, we get that
    \[
    \begin{matrix}
        \Theta_{x_i}= s_2u_i & \Theta_{x_ix_j}= 0 & \Theta_{x_iu}= 0&\Theta_{x_iu_j}=s_2\delta_{ij}\\
        \Theta_u = -2s_2 & \Theta_{ux_j} = 0& \Theta_{uu} = 0 &\Theta_{uu_j}=0\\
        \Theta_{u_i} = s_2x_i & \Theta_{u_ix_j}= s_2\delta_{ij} & \Theta_{u_iu}= 0 & \Theta_{u_iu_j}= 0.
    \end{matrix}
    \]
    Hence \eqref{psiulambdas1} becomes zero and \eqref{ldgyoungs} becomes
    \begin{equation*}
        \sum_{k=1}^n\frac{s_2}{1+ \lambda_k^2}\left(x_k + u_k\lambda_k\right)\pd_k\log\sqrt{\det g}.
    \end{equation*}
    Applying Young's inequality and simplifying, we get
    \begin{equation*}
        \Delta_g \log\sqrt{\det g} - \frac{1}{2n}|\nabla_g \log\sqrt{\det g}|^2_g \geq -\frac{ns_2^2}{2}\left(|x_0|^2 + |Du(x_0)|^2\right) \geq -C.
    \end{equation*}
\end{proof}

\begin{corollary}\label{tcor}
    Let $u$ be a $C^4$ convex solution to \eqref{tran} in $B_2(0)\subset\re^n$. Assuming the Hessian $D^2u$ is diagonalized at $x_0\in B_1(0)$, \eqref{J} holds with $C= C(n,t_2,t_3)$.
\end{corollary}
\begin{proof}
    As $\Theta(x,u(x),Du(x)) = t_1 + t_2\cdot x + t_3\cdot Du(x)$, we get 
    \[
    \Theta_{x_i} = t_{2,i} \quad\text{ and }\quad \Theta_{u_i} = t_{3,i}
    \]
    where all the remaining derivatives are zero.
    Hence \eqref{psiulambdas1} is zero and \eqref{ldgyoungs} becomes
    \begin{equation*}
        \sum_{k=1}^n\frac{1}{1+ \lambda_k^2}\left(t_{3,k} + t_{2,k}\lambda_k\right)\pd_k\log\sqrt{\det g}.
    \end{equation*}
    Applying Young's inequality and simplifying, we get
    \begin{equation*}
        \Delta_g \log\sqrt{\det g} - \frac{1}{2n}|\nabla_g \log\sqrt{\det g}|^2_g \geq -\frac{n}{2}\left(|t_2|^2 + |t_3|^2\right)= -C.
    \end{equation*}
\end{proof}

\begin{corollary}\label{rcor}
    Let $u$ be a $C^4$ convex solution to \eqref{rotator} in $B_2(0)\subset\re^n$. Assuming the Hessian $D^2u$ is diagonalized at $x_0\in B_1(0)$, \eqref{J} holds with $C= C(n,r_2)(1 + (\osc_{B_2}(u))^2)$.
\end{corollary}
\begin{proof}
    Let $x_0\in B_1$. As $\Theta(x,u(x),Du(x)) = r_1 + \frac{r_2}{2}(|x|^2 + |Du(x)|^2)$, we get 
    \[
    \begin{matrix}
        \Theta_{x_i}= r_2x_i & \Theta_{x_ix_j}= r_2\delta_{ij} &\Theta_{x_iu_j}=0\\
        \Theta_{u_i} = r_2u_i & \Theta_{u_ix_j}= 0  & \Theta_{u_iu_j}= r_2\delta_{ij}.
    \end{matrix}
    \]
    Then \eqref{psiulambdas1} and \eqref{ldgyoungs} are bounded below by
    \begin{align*}
        \sum_{a=1}^n\frac{\lambda_a}{1+\lambda_a^2}\bigg[r_2 +&r_2\lambda_a^2\bigg]+\sum_{k=1}^n\frac{r_2}{1+ \lambda_k^2}\left(u_k - x_k\lambda_k\right)\pd_k\log\sqrt{\det g}\\
        &\geq \sum_{k=1}^n\frac{r_2}{1+ \lambda_k^2}\left(u_k - x_k\lambda_k\right)\pd_k\log\sqrt{\det g}
    \end{align*}
    since $r_2\geq 0$ and $\lambda_a\geq0$ for all $1\leq a\leq n$. Thus, using Young's inequality and simplifying, we get
    \begin{equation*}
        \Delta_g \log\sqrt{\det g} - \frac{1}{2n}|\nabla_g \log\sqrt{\det g}|^2_g \geq -\frac{nr_2^2}{2}\left(|x_0|^2 + |Du(x_0)|^2\right) \geq -C.
    \end{equation*}
\end{proof}

\begin{lemma}\label{IJlem}
    Let $u$ be a $C^4$ convex solution of \eqref{s},\eqref{tran},\eqref{rotator},\eqref{slag} on $B_{2}(0)\subset\mathbb{R}^n$. Let 
    \begin{equation*}
        b= \log V = \log \sqrt{\det g}.
    \end{equation*}
    Then $b$ is $C^2$, and hence, for all nonnegative $\phi \in C_0^\infty(B_1)$, $b$ satisfies the integral Jacobi inequality, each with their respective constant $C$:
    \begin{equation*}
        \int_{B_1}-\langle \nabla_g \phi,\nabla_g b\rangle_g dv_g\geq c(n)\int_{B_1}\phi|\nabla_gb|^2dv_g-\int_{B_1}C\phi\; dv_g. 
    \end{equation*}
    Consequently, we have
    \begin{equation*}
        \int_{B_r}|\nabla_g b|^2 dv_g \leq C(n)\left(\frac{1}{1-r} + C\right)\int_{B_1}d v_g
    \end{equation*}
    for $0< r < 1$.
\end{lemma}
\begin{proof}
    Since $u$ is $C^4$, it follows that $g = I + (D^2u)^2$ is $C^2$. Note that $\det g$ is $C^2$ since the determinant is a smooth function, and furthermore, at each point, we have $\det g(x) =\prod_i^n(1 + \lambda_i^2(x))  \geq 1$. From this, it follows that $\log\sqrt{\det g}$ is well defined and $C^2$ as a composition of smooth and $C^2$ functions. It immediately follows, using \eqref{J} and integration by parts,
    \begin{align*}
        \int_{B_1}-\langle \nabla_g \phi,\nabla_g b\rangle_g dv_g &= \int_{B_1}\phi \Delta_g b\; dv_g\\
        &\geq c(n)\int_{B_1}\phi|\nabla_gb|^2dv_g-\int_{B_1}C\phi\; dv_g.
    \end{align*}

    Rearranging, we see that for any cutoff $\phi \in C_0^\infty(B_1)$,
    \begin{align*}
        \int_{B_1} \phi^2|\nabla_g b|^2\;dv_g &\leq \frac{1}{c(n)}\int_{B_1}\phi^2 \Delta_g b\; dv_g +\frac{1}{c(n)}\int_{B_1}\phi^2C\;dv_g \\
        &= -\frac{1}{c(n)}\int_{B_1}\langle 2\phi\nabla_g \phi,\nabla_g b\rangle_g dv_g + \frac{1}{c(n)}\int_{B_1}\phi^2C\;dv_g\\
        &\leq \frac{1}{2}\int_{B_1}\phi^2|\nabla_g b|^2dv_g + \frac{2}{c(n)^2}\int_{B_1}|\nabla_g \phi|^2 dv_g + \frac{1}{c(n)}\int_{B_1}\phi^2C\;dv_g. 
    \end{align*}
    Let $0< r< 1$. Choose $0\leq \phi \leq 1$ with $\phi=1$ on $B_r$ and $|D\phi|\leq \frac{2}{1-r}$ in $B_1$ to get
    \begin{align*}
        \int_{B_r}|\nabla_g b|^2dv_g&\leq \int_{B_1}\phi^2|\nabla_g b|^2dv_g\\
        &\leq \frac{4}{c(n)^2}\int_{B_1}|\nabla_g\phi|^2dv_g + \frac{2}{c(n)}\int_{B_1} \phi^2 C\; dv_g\\
        &\leq C(n)\left(\frac{1}{1-r} + C\right)\int_{B_1}dv_g.
    \end{align*}
\end{proof}

\section{Sobolev Inequalities and the Lewy-Yuan rotation}\label{LY}

We first recall the Lewy-Yuan rotation developed in \cite[p.122]{YY} for the convex potential $u$ of the Lagrangian graph $X = (x,Du(x))$: We
rotate it to $X = (\bar{x},D\bar{u}(\bar{x}))$ in a new co-ordinate system of $\re^n\times \re^n\cong \com^n$ via $\bar{z} = e^{-i\frac{\pi}{4}}z$, where $z = x + iy$ and $\bar{z} = \bar{x} + i\bar{y}$. That is,
\begin{equation}\label{lyrot}
\begin{cases}
     \bar{x} = \frac{\sqrt{2}}{2}x + \frac{\sqrt{2}}{2}Du(x)\\
    \bar{y} = D\bar{u} = -\frac{\sqrt{2}}{2} x + \frac{\sqrt{2}}{2}Du(x) .
\end{cases}
\end{equation} 
   
We state the following proposition from \cite[Prop 3.1]{CWY} and \cite[p.122]{YY}.
\begin{proposition}
    Let $u$ be a $C^4$ convex function on $B_R(0)\subset \re^n$. Then the Lagrangian submanifold $X = (x,Du(x))\subset \re^n\times \re^n$ can be represented as a gradient graph $X = (\bar{x},D\bar{u}(\bar{x}))$ of the new potential $\bar{u}$ in a domain containing a ball of radius
    \begin{equation}\label{rbar}
        \bar{R}\geq \frac{\sqrt{2}}{2}R
    \end{equation}
    such that in these coordinates the new Hessian satisfies
    \begin{equation}\label{dbaru}
        -I \leq D^2\bar{u} \leq I.
    \end{equation}
\end{proposition}

We define
\begin{equation*}
    \bar{\Omega}_r = \bar{x}(B_r(0)).
\end{equation*}
From \eqref{lyrot}, for $\bar{x}\in \bar{\Omega}_r$, we have that
\begin{equation}\label{rho}
    |\bar{x}|\leq r\frac{\sqrt{2}}{2} + ||Du||_{L^\infty(B_r)}\frac{\sqrt{2}}{2} := \rho(r),
\end{equation}
and from \eqref{rbar}, we have
\begin{equation*}
    \text{dist}(\bar{\Omega}_1,\pd\bar{\Omega}_{2n})\geq \frac{2n-1}{\sqrt{2}}\geq \frac{3}{\sqrt{2}}>2.
\end{equation*}

From \eqref{dbaru}, it follows that the induced metric on $X= (\bar{x},D\bar{u}(\bar{x}))$ in $\bar{x}-$coordinates is bounded by
\begin{equation}\label{barmetric}
    d\bar{x}^2 \leq g(\bar{x})\leq 2 d\bar{x}^2.
\end{equation}
Next, we state the following Sobolev inequality, which is a generalization of  Proposition 3.2 from \cite{CWY}. For the sake of completeness, we add a proof below. 
\begin{proposition}\label{sobo}
    Let $u$ be a $C^4$ convex function on $B_{R'}(0)\subset\re^n$. Let $f$ be a $C^2$ positive function on the Lagrangian surface $X=(x,Du(x))$. Let $0 < r < R < R'$ be such that $R-r> 2\sqrt{2}\epsilon$. Then
    \begin{equation*}
        \left[\int_{B_r}|(f-\tilde{f})^+|^\frac{n}{n-1}dv_g \right]^\frac{n-1}{n} \leq C(n)\left(\frac{\rho^2}{r \epsilon}\right)^{(n-1)}\int_{B_R}|\nabla_g(f - \tilde{f})^+|dv_g
    \end{equation*}
    where $\rho = \rho(R')$ is as defined in \eqref{rho}, and
    \[
    \tilde{f} = \frac{2}{|B_r|}\int_{B_R(0)} fdx.
    \]
\end{proposition}

We first state and prove a generalization of Lemma 3.2 from \cite{CWY}.
\begin{lemma}\label{iso}
    Let $\Omega_1\subset\Omega_2\subset B_{\rho}\subset\re^n$ and $\epsilon>0$. Suppose that dist$(\Omega_1,\pd\Omega_2)\geq 2\epsilon$; $A$ and $A^c$ are disjoint measurable sets such that $A\cup A^c = \Omega_2$. Then
    \[
        \min\{|A\cap\Omega_1|,|A^c\cap \Omega_1|\}\leq C(n)\frac{\rho^n}{\epsilon^n}|\pd A\cap \pd A^c|^\frac{n}{n-1}.
    \]
\end{lemma}
\begin{proof}
   Define the following continuous function on $\Omega_1$:
   \[
   \xi(x) = \frac{|A\cap B_{\epsilon}(x)|}{|B_\epsilon|}.
   \]
   Case $1$. $\xi(x_0)=\frac{1}{2}$ for some $x_0\in \Omega_1$. We then have that $B_\epsilon(x_0)\subset\Omega_2$ by dist$(\Omega_1,\pd\Omega_2)\geq 2\epsilon$. From the classical relative isoperimetric inequality for balls \cite[Theorem 5.3.2]{LinYang}, we have
   \begin{align*}
       \frac{|B_\epsilon|}{2} &= |A\cap B_\epsilon(x_0)|\\
       &\leq C(n)|\pd(A\cap B_\epsilon(x_0))\cap\pd(A^c\cap B_\epsilon(x_0))|^\frac{n}{n-1}\\
       &\leq C(n)|\pd A\cap \pd A^c|^\frac{n}{n-1}.
   \end{align*}
   Hence,
   \[
   \min\{|A\cap\Omega_1|,|A^c\cap \Omega_1|\}\leq|\Omega_1|\leq|B_\rho|=\frac{\rho^n}{\epsilon^n}|B_\epsilon|\leq C(n)\frac{\rho^n}{\epsilon^n}|\pd A\cap \pd A^c|^\frac{n}{n-1}.
   \]
   Case $2$. $\xi(x)>\frac{1}{2}$ for all $x\in\Omega_1$. Cover $\Omega_1$ by $N\leq C(n)\frac{\rho^n}{\epsilon^n}$ balls of radius epsilon $B_\epsilon(x_i)$ for some uniform constant $C(n)$ since $\Omega_1$ is bounded. Note that all of these balls are in $\Omega_2$ since dist$(\Omega_1,\pd\Omega_2)\geq2\epsilon$. Thus,
   \[
   |A^c\cap B_\epsilon(x_i)|= \min\{|A\cap B_\epsilon(x_i)|,|A^c\cap B_\epsilon(x_i)|\} \leq C(n)|\pd A\cap \pd A^c|^\frac{n}{n-1}.
   \]
   Summing over the cover, we get
   \[
   |A^c\cap \Omega_1|\leq \sum_{i=1}^N|A^c\cap B_\epsilon(x_i)|\leq C(n)\frac{\rho^n}{\epsilon^n}|\pd A\cap \pd A^c|^\frac{n}{n-1}.
   \]
   Case $3$. $\xi(x)<\frac{1}{2}$ for all $x\in \Omega_1$. Repeating the same proof in Case $2$, but with $A$ instead of $A^c$, yields the same result.
\end{proof}

\begin{proof}[Proof of Proposition \ref{sobo}]
Let $M=||f||_{L^\infty(B_r)}$. If $M\leq \ti{f}$, then $(f-\ti{f})^+=0$ on $B_r$, and hence, the left hand side is zero, from which the result follows immediately. We assume $\ti{f}< M$. By the Morse-Sard Lemma \cite[Lemma 13.15]{maggi},\cite{sard}, $\{x | f(x) = t\}$ is $C^1$ for almost all $t \in (\ti{f},M)$. We first show that for such $t$,
\begin{equation}\label{fiso}
   |\{x | f(x)> t\}\cap B_r|_g\leq C(n)\frac{\rho^{2n}}{r^n\epsilon^{n}}|\{x | f(x) = t\}\cap B_R|_g^\frac{n}{n-1}. 
\end{equation}
Note $|\cdot|_g$ is the metric with respect to $g$, and $|\cdot|$ is the Euclidean metric.

Let $t>\ti{f}$. It must be that
\[
    \frac{|B_r|}{2} > |\{x | f(x) > t\}\cap B_r|
\]
since otherwise
\[
M = \frac{2}{|B_r|}\int_0^M\frac{|B_r|}{2}dt\leq \frac{2}{|B_r|}\int_0^M|\{x|f(x)>t\}\cap B_r|dt \leq \frac{2}{|B_r|}\int_{B_R}fdx=\ti{f} < M.
\]
From this, it follows
\begin{equation}\label{Atrot}
    |\{x | f(x) \leq t\}\cap B_r| > \frac{|B_r|}{2}.
\end{equation}

Let $A_t = \{\bar{x}|f(\bar{x}) > t\}\cap\bar{\Omega}_R$. From Lemma \ref{iso}, we have that
\[
\min\{|A_t\cap \bar{\Omega}_r|,|A_t^c\cap\bar{\Omega}_r|\}\leq C(n)\frac{\rho^n}{\epsilon^n}|\pd A_t\cap \pd A_t^c|^\frac{n}{n-1}.
\]

If $|A_t\cap \bar{\Omega}_r|\leq |A_t^c\cap \bar{\Omega}_r|$, then
\begin{align}
    |A_t\cap \bar{\Omega}_r|_{g(\bar{x})}&\leq 2^\frac{n}{2}|A_t\cap \bar{\Omega}_r|\nonumber\\
    &\leq C(n)\frac{\rho^n}{\epsilon^n}|\pd A_t\cap \pd A_t^c|^\frac{n}{n-1}_{g(\bar{x})}.\nonumber
\end{align}

On the other hand, if $|A_t\cap \bar{\Omega}_r|> |A_t^c\cap \bar{\Omega}_r|$, from \eqref{Atrot}, we have
\[
|A_t^c\cap\bar{\Omega}_r|>\frac{|B_r|}{2^{n+1}},
\]
and so
\[
|A_t\cap\bar{\Omega}_r|\leq \frac{\rho^n}{r^n}|B_r|\leq 2^{n+1}\frac{\rho^n}{r^n}|A_t^c\cap\bar{\Omega}_r|.
\]
Therefore
\[
|A_t\cap\bar{\Omega}_r|_{g(\bar{x})}\leq C(n)\frac{\rho^n}{r^n}|A_t^c\cap\bar{\Omega}_r|\leq C(n)\frac{\rho^{2n}}{r^n\epsilon^n}|\pd A_t\cap\pd A_t^c|^\frac{n}{n-1}_{g(\bar{x})}.
\]

In either case, we have
\[
|A_t\cap\bar{\Omega}_r|_{g(\bar{x})}\leq C(n)\frac{\rho^{2n}}{r^n\epsilon^n}|\pd A_t\cap\pd A_t^c|^\frac{n}{n-1}_{g(\bar{x})},
\]
which in our original coordinates is \eqref{fiso}.

We get
\begin{align*}
    \bigg[\int_{B_r}&|(f-\ti{f})^+|^\frac{n}{n-1}dv_g\bigg]^\frac{n-1}{n}\\
    &=\left[\int_0^{M-\ti{f}}|\{x | f(x) - \ti{f}> t\}\cap B_r|_g dt^\frac{n}{n-1}\right]^\frac{n-1}{n} \text{ via Layer cake \cite[Ex 1.13]{maggi}}\\
    &\leq \int_0^{M-\ti{f}}|\{ x | f(x) - \ti{f}> t\}\cap B_r|^\frac{n-1}{n}_g dt \text{ via the H-L-P inequality \cite[(5.3.3)]{LinYang}}\\
    &\leq C(n)\left(\frac{\rho^2}{r\epsilon}\right)^{n-1}\int_{\ti{f}}^M|\{x| f(x) = t\}\cap B_R|_g dt \text{ via \eqref{fiso}}\\
    &\leq C(n)\left(\frac{\rho^2}{r\epsilon}\right)^{n-1}\int_{B_R}|\nabla_g(f - \ti{f})^+|dv_g \text{ via the co-area formula \cite[Thm 4.2.1]{LinYang}}
\end{align*}
which completes the proof.
\end{proof}

\section{Proof of the main Theorems}
We now prove Theorem \ref{main1} from which Theorem \ref{main0} follows.
\begin{proof}[Proof of Theorem \ref{main1}]
For simplifying notation in the remaining proof, we assume $R=2n+2$ and $u$ is a solution on $B_{2n+2}\subset\mathbb{R}^n$. Then by scaling $v(x)=\frac{u(\frac{R}{2n+2}x)}{(\frac{R}{2n+2})^2}$, we get the estimate in Theorem \ref{main1}. The proof follows in the spirit of \cite[Section 3]{CWY}. Under our assumption $|\Theta|\geq (n-1)\frac{\pi}{2}$, we have that $u$ is convex. Note $C=C(n,\nu_1,\nu_2)(1 + (\osc_{B_{2n+2}}(u))^2)$
is the positive constant from \eqref{J}.
\begin{itemize}

\item[Step 1.] We use the rotated Lagrangian graph $X = (\bar{x},D\bar{u}(\bar{x}))$ via the Lewy-Yuan rotation, as illuatrated in Section \ref{LY}. Consider $b = \log V$ on the manifold $X= (x,Du(x))$, where $V$ is the volume element in the original coordinates. In the rotated coordinates $b(\bar{x})= \log V(\bar{x})$ satisfies 
\begin{align}
    \bigg(g^{ij}(\bar{x})\frac{\pd^2}{\pd\bar{x}_i\pd\bar{x_j}}-& g^{jp}(\bar{x})\frac{\pd \Theta(x(\bar{x}),u(x(\bar{x})), \frac{\sqrt{2}}{2}\bar{x} + \frac{\sqrt{2}}{2}D\bar{u}(\bar{x}))}{\pd\bar{x}_q}\frac{\pd^2 \bar{u}(\bar{x})}{\pd \bar{x}_q\pd\bar{x}_p} \frac{\pd}{\pd \bar{x}_j}\bigg) b(\bar{x})\nonumber\\
    &= \Delta_{g(\bar{x})}b(\bar{x}) \geq - C. \label{subsol}
\end{align}
The nondivergence and divergence elliptic operator are both uniformly elliptic due to \eqref{dbaru}.

From \eqref{lyrot}, we have
\begin{equation*}
\begin{cases}
     x(\bar{x}) = \frac{\sqrt{2}}{2}\bar{x} - \frac{\sqrt{2}}{2}D\bar{u}(\bar{x})\\
     Du(x(\bar{x})) = \frac{\sqrt{2}}{2} \bar{x} + \frac{\sqrt{2}}{2}D\bar{u}(\bar{x})
\end{cases}
\end{equation*} 
from which it follows that
\begin{align}
    &\frac{\pd \Theta(x(\bar{x}),u(x(\bar{x})), \frac{\sqrt{2}}{2}\bar{x} + \frac{\sqrt{2}}{2}D\bar{u}(\bar{x}))}{\pd\bar{x}_q}\nonumber \\
    &= \sum_{j=1}^n\Theta_{x_j}\frac{\pd x_j}{\pd \bar{x}_q} + \Theta_u\sum_{j=1}^n u_j\frac{\pd x_j}{\pd\bar{x}_q} + \sum_{j=1}^n\Theta_{u_j}\frac{\pd}{\pd\bar{x}_q}\left(\frac{\sqrt{2}}{2}\bar{x}_j + \frac{\sqrt{2}}{2}\bar{u}_j\right)\nonumber\\
    &=\frac{\sqrt{2}}{2}(\Theta_{x_q} + \Theta_u u_q)( 1- \bar{\lambda}_q) + \frac{\sqrt{2}}{2}\Theta_{u_q}(1 + \bar{\lambda}_q)\nonumber\\
    &\leq \sqrt{2}\nu_1(1 + \osc_{B_{2n+2}}(u)).\label{D1bound}
\end{align}

Denote
\[
\tilde{b} = \frac{2}{|B_1(0)|}\int_{B_{2n}(0)}\log V dx.
\]
Via the local mean value property of nonhomogeneous subsolutions \cite[Theorem 9.20]{GT} (see Appendix Theorem \ref{locmvp}), we get the following, from \eqref{subsol} and \eqref{D1bound}:
\begin{align*}
    (b - \tilde{b})^+(0) &= (b - \tilde{b})^+(\bar{0})\\
    &\leq C(n)\left[\ti{C}^{\;n-1}\left(\int_{B_{1/\sqrt{2}}(\bar{0})}|(b - \tilde{b})^+(\bar{x})|^\frac{n}{n-1}d\bar{x} \right)^\frac{n-1}{n} + C\left(\int_{B_{1/\sqrt{2}}(\bar{0})}d\bar{x}\right)^\frac{1}{n}\right]\\
    &\leq C(n)\left[\ti{C}^{\;n-1}\left(\int_{B_{1/\sqrt{2}}(\bar{0})}|(b - \tilde{b})^+(\bar{x})|^\frac{n}{n-1}dv_{g(\bar{x})} \right)^\frac{n-1}{n} + C\left(\int_{B_{1/\sqrt{2}}(\bar{0})}dv_{g(\bar{x})}\right)^\frac{1}{n}\right]\\
    &\leq C(n)\left[\ti{C}^{\;n-1}\left(\int_{B_1(0)}|(b - \tilde{b})^+(x)|^\frac{n}{n-1}dv_{g(x)} \right)^\frac{n-1}{n} + C\left(\int_{B_1(0)}dv_g\right)^\frac{1}{n}\right]
\end{align*}
where $\ti{C}=(1 + \nu_1 + \nu_1\osc_{B_{2n+2}}(u))$ and $C=C(n,\nu_1,\nu_2)(1 + (\osc_{B_{2n+2}}(u))^2)$ is the positive constant from \eqref{J}.

The above mean value inequality can also be derived using the De Giorgi-Moser iteration \cite[Theorem 8.16]{GT}.

\item[Step 2.] By Proposition \ref{sobo} with $\rho = \rho(2n+1)$ and Lemma \ref{IJlem}, we have
\begin{align}
    b(0) &\leq C(n)\ti{C}^{\;n-1}\rho^{2(n-1)}\int_{B_{2n}}|\nabla_g(b - \tilde{b})^+|dv_g + CC(n)\left(\int_{B_{2n}}Vdx\right)^\frac{1}{n}+ C(n)\int_{B_{2n}}\log V dx\nonumber\\
    &\leq C(n)\ti{C}^{\;n-1}\rho^{2(n-1)}\left(\int_{B_{2n}}|\nabla_gb|^2dv_g\right)^\frac{1}{2}\left(\int_{B_{2n}}Vdx\right)^\frac{1}{2} \nonumber\\
    &\hspace{1in}+CC(n)\left(\int_{B_{2n}}Vdx\right)^\frac{1}{n}+ C(n)\int_{B_{2n}} V dx\nonumber\\
    &\leq C(n)(1 + \ti{C}^{\;n-1}(1 + C)^\frac{1}{2})\rho^{2(n-1)}\int_{B_{2n+1}}V dx+ CC(n)\left(\int_{B_{2n+1}}Vdx\right)^\frac{1}{n}. \label{b0}
\end{align}

\item[Step 3.] We bound the volume element using the rotated coordinates. From \eqref{barmetric}, we have
\[
Vdx = \bar{V}d\bar{x} \leq 2^\frac{n}{2}d\bar{x}.
\]
Since $\bar{\Omega}_{2n+1} = \bar{x}(B_{2n+1}(0))$, we get
\[
\int_{B_{2n+1}}Vdx= \int_{\bar{\Omega}_{2n+1}}\bar{V}d\bar{x} \leq 2^\frac{n}{2}\int_{\bar{\Omega}_{2n+1}}d\bar{x}\leq C(n)\rho^n.
\]
Hence, from \eqref{b0}, we get
\begin{equation}\label{b02}
    b(0)\leq C(n)(1 + \ti{C}^{\;n-1}(1 + C)^\frac{1}{2})\rho^{3n-2} + CC(n)\rho\leq C(n)(1 + \ti{C}^{\;n-1}(1 + C)^\frac{1}{2} + C)\rho^{3n-2}.
\end{equation}
By plugging in \eqref{rho}, $\ti{C}$, and $C$, and using
\[
(a + b)^p \leq 2^p(a^p + b^p),\quad\text{for } a,b\geq 0, p>0,
\]
as well as Young's inequality, we have
\begin{align}
    C(n)&(1 +\ti{C}^{\;n-1}(1 + C)^\frac{1}{2} + C)\rho^{3n-2}\nonumber\\
    &\leq C(n,\nu_1,\nu_2)(1 + (\osc_{B_{2n+2}}(u))^{n-1}+ (\osc_{B_{2n+2}}(u))^{n}\nonumber\\
    &\hspace{1in}+(\osc_{B_{2n+2}}(u))^2)(1 + (\osc_{B_{2n+2}}(u))^{3n-2})\nonumber\\
    &\leq C(n,\nu_1,\nu_2)(1 + (\osc_{B_{2n+2}}(u))^{4n-2}).\label{Cosc}
\end{align}

By combining \eqref{b02} and \eqref{Cosc} and exponentiating, we get 
\[ 
|D^2 u(0)|\leq C_1\exp[C_2(\osc_{B_{2n+2}}(u))^{4n-2}]
\] 
where $C_1$ and $C_2$ are positive constants depending on $\nu_1,\nu_2$, and $n$.
\end{itemize}

\end{proof}

\begin{proof}[Proof of Theorem \ref{main0}]
Repeating the above proof, but with the constant $C$ for equations \eqref{s} and \eqref{rotator} from Corollaries \ref{scor} and \ref{rcor} respectively, we get the desired estimate. Note, in the case of \eqref{tran}, we get $C = \ti{C}= C(n,t_2,t_3)$, and so \eqref{b02} becomes
\[
b(0)\leq C(n,t_2,t_3)\rho^{3n-2}
\]
resulting in the estimate
\[
|D^2 u(0)|\leq C_1\exp[C_2(\osc_{B_{2n+2}}(u))^{3n-2}]
\]
where $C_1$ and $C_2$ depend on $n,t_2,t_3$.
\end{proof}

\begin{remark}\label{dbvp}
    We prove analyticity of a $C^0$ viscosity solution within its domain by outlining a modification of the approach in \cite[Section 4]{CWY}. Note, we obtain smooth approximations via \cite[Theorem 4]{CNS}, \cite{Trudinger}. Let
    \[
    F(x,u,Du,D^2u) = G(D^2u) - \Theta(x,u,Du) = \sum_{j=1}^n \arctan\lambda_j - \Theta(x,u,Du).
    \]
    We wish to apply Evans-Krylov-Safonov theory (\cite[Theorem 17.15]{GT}) which requires $F(x,z,p,r)$ to be concave in $z,p,r$ and the following structure conditions to hold
    \begin{align*}
    &0< \ell|\xi|^2 \leq F_{ij}(x,z,p,r)\xi_i\xi_j\leq \Lambda|\xi|^2,\\
    &|F_p|,|F_z|,|F_{rx}|,|F_{px}|,|F_{zx}|\leq \mu\ell,\\
    &|F_x|,|F_{xx}|\leq \mu\ell(1 + |p| + |r|),
    \end{align*}
    for all nonzero $\xi \in\re^n$, where $\ell$ is a nonincreasing function of $|z|$, and $\Lambda$ and $\mu$ are nondecreasing functions of $|z|$. Note, for our operator $F$ defined above, $F_{rx}=0$. 
    
    We have that $G(D^2u)$ is concave, and by our assumption, $\Theta(x,z,p)$ is partially convex in $p$. By additionally assuming partial convexity of $\Theta$ in $z$, we get that $F$ is concave in $z,p,r$ as desired. Note, for equations \eqref{s},\eqref{tran},\eqref{rotator}, this condition is naturally satisfied. 
    
    Theorems \ref{main0} and \ref{main1} give us that
    \[
    0 < \frac{1}{1 + [C(\osc_{B_R}(u))]^2}|\xi|^2 \leq F_{ij}(x,z,p,r)\xi_i\xi_j\leq |\xi|^2.
    \]
    Taking $\ell = \frac{1}{1 + C^2}$ and $\mu = \frac{\nu_1 + \nu_2}{\ell}$, we see that the other conditions are satisfied. Hence, we achieve a $C^{2,\alpha}$ bound. By applying classical elliptic theory \cite[Lemma 17.16]{GT} and \cite[p202]{morrey}, to solutions of \eqref{s},\eqref{tran},\eqref{rotator} we get the analyticity of $u$. 
\end{remark}

\section{Appendix}
Our proof requires an explicit dependence of the constants appearing in Theorem 9.20 of \cite{GT} on the oscillation of the potential, when applied to \eqref{subsol}. We state and prove an adaptation of \cite[Theorem 9.20]{GT} to our specific case.

First, we clarify some notations and terminology. We have
\begin{align*}
   L &= a^{ij}(\bar{x})\frac{\pd}{\pd \bar{x}_i\pd \bar{x}_j} + b^j(\bar{x})\frac{\pd}{\pd \bar{x}_j} \\
   &= g^{ij}(\bar{x})\frac{\pd}{\pd \bar{x}_i\pd \bar{x}_j} -g^{jp}(\bar{x})\frac{\pd \Theta(x(\bar{x}),u(x(\bar{x})), \frac{\sqrt{2}}{2}\bar{x} + \frac{\sqrt{2}}{2}D\bar{u}(\bar{x}))}{\pd\bar{x}_q}\frac{\pd^2 \bar{u}(\bar{x})}{\pd \bar{x}_q\pd\bar{x}_p} \frac{\pd}{\pd \bar{x}_j}.
\end{align*}
From this and \eqref{dbaru}, it follows that $\frac{1}{2}|\xi|^2 \leq a^{ij}(\bar{x})\xi_i\xi_j\leq |\xi|^2$, and we have from \eqref{D1bound}:
\[
|b|\leq \sqrt{2n}\nu_1(1 + \osc_{B_{2n+2}}(u)).
\]
By $\Omega$, we denote a $C^{1,1}$ domain in $\mathbb R^n$.
\begin{theorem}\label{locmvp}
    Let $u\in C^2(\Omega)\cap W^{2,n}(\Omega)$ and suppose that $Lu\geq f$, where $f\in L^n(\Omega)$. Then for any ball $B= B_{2R}(y)\subset\Omega$, we have
    \[
    \sup_{B_R(y)} u\leq C(n)\left\{(RC)^{n-1}||u^+||_{L^\frac{n}{n-1}(B)} + R||f||_{L^n(B)}\right\}
    \]
    where $C=(1 + \nu_1 + \nu_1\osc_{B_{2n+2}}(u))$.
\end{theorem}
\begin{proof}
    Without loss of generality, we assume that $B= B_1(0)$, the general case is recovered via $x\to (x-y)/2R$. For $\beta = 2(n-1)$, we define the cutoff function $\eta$ by
    \[
    \eta(x)(1 - |x|^2)^\beta.
    \]
    Differentiating, we get
    \begin{align*}
        D_i\eta &= -2\beta x_i(1 - |x|^2)^{\beta-1},\\
        D_{ij}\eta & -2\beta\delta_{ij}(1-|x|^2)^{\beta-1} + 4\beta(\beta-1)x_ix_j(1-|x|^2)^{\beta - 2}.
    \end{align*}
    Set $v = \eta u$. We have
    \begin{align*}
        a^{ij}D_{ij}v &= \eta a^{ij}D_{ij}u +2 a^{ij}D_i\eta D_ju + ua^{ij}D_{ij}\eta\\
        &\geq \eta(f - b^iD_i u) + 2 a^{ij}D_i\eta D_ju + ua^{ij}D_{ij}\eta.
    \end{align*}
    Denote $\Gamma^+$ to be the upper contact set of $v$ in $B$. We have that $u > 0$ on $\Gamma^+$, and using the concavity of $v$ on $\Gamma^+$, we estimate
    \begin{align*}
        |Du| &= \frac{1}{\eta}|Dv - uD\eta|\\
        &\leq \frac{1}{\eta}\left(|Dv| + u|D\eta|\right)\\
        &\leq \frac{1}{\eta}\left(\frac{v}{1-|x|}+ u|D\eta| \right)\\
        &\leq 2(1 + \beta)\eta^{-1/\beta}u.
    \end{align*}
    Thus, on $\Gamma^+$, we have
    \begin{align*}
        -a^{ij}D_{ij}v &\leq [(16\beta^2+ 2\eta\beta)\eta^{-2/\beta} + 2\beta|b|\eta^{-1/\beta}]v + \eta f\\
        &\leq C\eta^{-2/\beta}v + f,
    \end{align*}
    where $C = C(n)(1 + \nu_1 + \nu_1\osc_{B_{2n+2}}(u))$. Applying \cite[Lemma 9.3]{GT}, we get
    \begin{align}
        \sup_B v&\leq C(n)\left|\left|\frac{a^{ij}D_{ij}v}{1/2}\right|\right|_{L^n(\Gamma^+)}\nonumber\\
        &\leq C(n)\left\{C||\eta^{-2/\beta}v^+||_{L^n(B)} + ||f||_{L^n(B)}\right\}\nonumber\\
        &\leq C(n)\left\{C(\sup_Bv^+)^{1-2/\beta}||(u^+)^{2/\beta}||_{L^n(B)} + ||f||_{L^n(B)}\right\}.\label{supv}
    \end{align}
    Let $q = \beta/2$, and so $p = 1/(1-2/\beta)$. Using Young's inequality and recalling $\beta = 2(n-1)$, we get that
    \begin{align*}
        (\sup_Bv^+)^{1-2/\beta}||(u^+)^{2/\beta}||_{L^n(B)} &\leq \epsilon (\sup_B v^+)^{p(1 - 2/\beta)} + \epsilon^{-q/p}||(u^+)^{2/\beta}||_{L^n(B)}^{\beta/2}\\
        &= \epsilon\sup_B v^+ + \epsilon^{2-n}||u^+||_{L^\frac{n}{n-1}(B)}.
    \end{align*}
    Plugging this into \eqref{supv}, we get 
    \[
    (1 - C(n)C\epsilon)\sup_B v\leq C(n)\left\{C\epsilon^{2-n}||u^+||_{L^\frac{n}{n-1}(B)} +||f||_{L^n(B)}\right\}.
    \]
    Let $\epsilon = \frac{1}{2C(n)C}$. We get
    \[
    \sup_B v\leq C(n)\left\{C(n)C^{\;n-1}||u^+||_{L^\frac{n}{n-1}(B)} + ||f||_{L^n(B)}\right\},
    \]
    from which our desired estimate follows.
\end{proof}

\bibliographystyle{amsalpha}
\bibliography{lib}

\end{document}